\documentclass{jaca}  
\usepackage{graphicx,xcolor} 

\newtheorem{Ass}{Assumption}
\def\beP{\begin{prop}} \def\eeP{\end{prop}}
\def\beT{\begin{thm}} \def\eeT{\end{thm}}
\def\beR{\begin{rem}} \def\eeR{\end{rem}}
\def\beL{\begin{lem}} \def\eeL{\end{lem}}
\def\beXa{\begin{exmp}} \def\eeXa{\end{exmp}}
 \def\beC{\begin{cor}}  \def\FiW{Fisher Wright}
   
\def\eeC{\end{cor}}
\newtheorem{Qu}{Problem}
\def\eeD{\end{defn}} \def\beD{\begin{defn}}
\def\beQ{\begin{Qu}}
  \def\eeQ{\end{Qu}}

\def\eqs{\begin{eqnarray*}}
\def\ens{\end{eqnarray*}}
\def\H{\widehat}
\def\eqa{\begin{eqnarray}}
\def\ena{\end{eqnarray}}

\def\BEN{\begin{enumerate}}  \def\BI{\begin{itemize}}
\def\EEN{\end{enumerate}}   \def\EI{\end{itemize}} \def\im{\item} \def\Lra{\Longrightarrow}  \def\eqr{\eqref}  
\long\def\symbolfootnote[#1]#2{
\begingroup
\def\thefootnote{\fnsymbol{footnote}}\footnote[#1]{#2}
\endgroup}
\def\fn{\symbolfootnote}

\def\g{\gamma}   
  \def\th{\theta} 
\def\e{\epsilon}  \def\se{\sqrt{\eps}}
\def\k{\kappa} \def\l{\lambda} \def\a{\alpha} \def\b{\beta}
  \def\Thr{Therefore, }
\usepackage{color}
\newcommand*{\la}{\langle}

  \def\ith{it holds that } 
\def\m{b}    \def\th{\theta}
   \def\s{\sigma} \def\F{\Phi} \def\f{\varphi} 
  \def\bc{\begin{cases} \def\fp{first passage }
  }    \def\wlg{w.l.o.g. }
\def\ec{\end{cases}} \def\expo{exponential } 
  \def\qu{\quad} \def\for{\forall} \def\T{\widetilde} \def\rv{random variable }
  \def\beq{\begin{eqnarray}} \def\eeq{\end{eqnarray}}
   \def\be{\begin{equation}} \def\ee{\end{equation}}
   \def\Eq{\Leftrightarrow}  \def\iec{it is easy to check }
     \def\proba{probability} \def\procs{processes}
   
   \newcommand{\pd}[2]{\frac{\partial #1}{\partial #2}}

\def\bea{\begin{multline*}}  \def\sats{satisfies }    \def\BD{birth and death processes} \def\br{branching }
\def\eea{\end{multline*}} \def\la{\label}  \def\mI{{\mathcal I}}    
 \def\LT{Laplace transform } 
\def\rv{random variable} 
\def\Fe{For example, } \def\sec{\section} \def\Frt{Furthermore, }

  \def\para{parameter }
  \def\R{{\mathbb R}} \def\FW{Freidlin-Wentzell}
  \def\I{\infty} \def\det{deterministic} \def\d{\delta}
\def\nne{nonnegative } \def\fr{\frac}   \def\t{\tau} 
   
 \def\sec{\section}    \def\xri{\xrightarrow} \def\fe{for example }  \def\wk{well known } 
 \def\fp{first passage }

\newcommand{\eps}{\varepsilon} 

  \def\ovl{\overline}  
  \def\app{approximation} 
 \def\satg{satisfying }

\newcommand{\figu}[3]{
\begin{figure}[!h]
\begin{center}
{\includegraphics[width=13 cm, height=8 cm]{#1}}
\end{center}

\vspace{-0.2cm}
\caption{\hspace{0.25cm}#2\label{f:#1}}
\end{figure}
}
    \def\Frt{Furthermore, }
\def\how{however }

 \newcommand{\LL}{\mathcal{L}}  
\def\ith{it holds that }  \def\se{\sqrt{\eps}}
\def\m{\mu}    \def\th{\theta}
\newcommand{\cd}{(\cdot)}

\begin{document}


 \title[Beyond Wentzell-Freidlin]{Beyond Wentzell-Freidlin: semi-deterministic approximations for diffusions
with small noise and  a repulsive critical boundary point}


\author{Florin Avram and Jacky Cresson}




\begin{abstract}
  We extend below a  limit theorem \cite{BCHK} for diffusion models used in population theory.
\end{abstract}

\KeysAndCodes{dynamical systems, small noise, linearization, semi-deterministic fluid approximation}{AMS 60J60}


\section{Introduction}\label{sec1}

 A diffusion with small noise is defined as the solution of
a stochastic differential equation (SDE) driven by  standard Brownian motion $B_t\cd$ (defined on a probability space and progressively measurable with respect to an increasing filtration)
\begin{equation}\label{main}
\bc dX^\eps_t =  \m(X^\eps_t) dt + \sqrt{\eps} \s(X^\eps_t) dB_t,\quad t\ge 0, 
\\
X^\eps_0=x_0=\eps,  X^\eps_t \in \mI:=(0,r) \ec
\end{equation}
where  $ 0< r\leq +\infty$, $\eps>0, \m:\mI \mapsto \R$, $\sigma:\mI \mapsto \R_{>0}$ and $\m,\s$ satisfy  conditions ensuring that \eqr{main} has a strong unique solution (for example,  $\m$ is locally Lifshitz and  $\s$ satisfies the Yamada-Watanabe conditions \cite[(2.13), Ch.5.2.C]{KS}).\fn[4]{For reviews discussing the existence of strong and weak solutions, see for example \cite{breiman1992probability,helland1996one,cherny2005singular}.}

When $\eps \to 0$, \eqr{main} is a small perturbation of the dynamical system/ordinary differential equation (ODE):
\begin{equation}\label{detmain}
\frac{dx_t}{dt}=\m(x_t),\quad t\ge 0,
\end{equation}
 which will also be supposed to admit a unique continuous solution $x_t, t \in \R_+$ subject to any $x_0\in (0,r)$, and the flow of which
will be denoted by $\phi_t(x)$.

A basic result in the field is  the ``fluid limit",  which  states that when \eqr{main} admits a strong unique solution, the effect of noise  is negligible as $\eps\to 0$, on any {\bf fixed time interval}  $[0,T]$:

\begin{thm} \label{t:Kurtz}[Freidlin and Wentzell] \cite[Thm 1.2, Ch. 2.1]{FW}
Let  $X^\varepsilon_t$  satisfy \eqr{main}, assume $\m,\s$ satisfy the Lifshitz condition,      and  that $X^\eps_0\xrightarrow[\eps\to 0]{\P} x_0\in \R_+$, where $\xrightarrow[\eps\to 0]{\P}$ denotes convergence in probability. Then, for any fixed $T$
\[
\sup_{t\leq T}|X^\varepsilon_t-x_t|\xrightarrow[\eps\to 0]{\P} 0,
\]
where $x_t$ is the solution of \eqref{detmain} subject to the initial condition $x_0$.\fn[5]{For other  deterministic limit theorems for  one-dimensional diffusions, see also  Gikhman and Skorokhod \cite{skorokhod2009asymptotic}, Freidlin and Wentzell \cite{FW},  Keller et al. \cite{keller1988asymptotic},  and Buldygin et al. \cite{buldygin2007prv}.}
\end{thm}

{ Although interesting, this result does not give any understanding of the asymptotic behavior of the diffusion process for  times converging to infinity;  in particular,  it does not tell us how the diffusion travels between equilibrium points (which requires times converging to infinity). Following \cite{barbour2015escape,BCHK}, we go here  beyond Theorem \ref{t:Kurtz}, by analyzing the way a diffusion process leaves an unstable equilibrium point. Precisely, we make the following assumptions:}

\begin{Ass} \la{a:A}
Suppose from now on that $l=0,\m(0)=0, \m'(0)>0$, which makes
zero an {\bf unstable equilibrium point of}  \eqref{detmain} and of \eqref{main}.
\end{Ass}

 {Note that under Assumption \ref{a:A}, the \FW \; theorem \ref{t:Kurtz} implies that the
solution of \eqref{main} started from a small positive initial condition $X^\eps_0=\eps>0$ converges to zero
on any fixed bounded interval
\[
\sup_{t\leq T}\big|X_t^\eps\big| \xrightarrow[\eps\to 0]{\P} 0, \qquad \forall T\ge 0.
\]

\begin{Ass} \la{a:B}
Put now $a(x)=\s^2(x)$,   and assume
that  $a(0)=\s(0)=0, a'(0)>0$, which makes $0$ a
singular point  of the diffusion \eqref{main}-- see for example \cite{cherny2005singular}.
\end{Ass}

\beR Note that $a'(0)>0$ rules out important population theory models like the
 linear Gilpin Ayala
 diffusion \cite{liu2015sufficient} with \begin{equation} \la{lGA} \m(x)=\g x \Big(1  -  \; (\fr{x}{x_c})^\a \Big), \s(x)=\sqrt{\eps} x \Eq a(x) =\eps x^2, \g >0, x_c>0, \a >0,\end{equation} which includes by setting  $\a=1$ another favorite, the logistic-type Verlhurst-Pearl diffusion \cite{Vallois,evans2015protected,hening2018optimal}.
\eeR

Recently,  a new type of limit theorem  \cite{BCHK} was discovered when $T \to \I$ under Assumptions  \ref{a:A}, \ref{a:B}, when $x_0^\eps $ converges to the unstable { equilibrium} point  of  \eqr{detmain}.
Following \cite{BCHK}, let \begin{equation} T^\eps := \dfrac 1 {\m'(0)} \log \dfrac 1 {\eps} \la{T}\end{equation}denote the solution of the equation $ \phi_{t,lin} (x_0)= x_0 e^{\m'(0) t}=1$ {  where $\phi_{t,lin} (x_0)$ is the flow of the linearized system of \eqref{detmain} in $0$}, and divide
the evolution of the process in three time-intervals:
\begin{equation} [0,t_c:=c T^\eps],[t_c, t_1:=T^\eps],[ t_1,\I), c \in (1/2,1)\end{equation}(the restriction $c>1/2$ is used in  \eqr{Gro}).

 It turns out   that this partition allows separating   the life-time of  diffusions with small noise, exiting an unstable point of the fluid limit,   into  three periods with distinct behaviors:
\figu{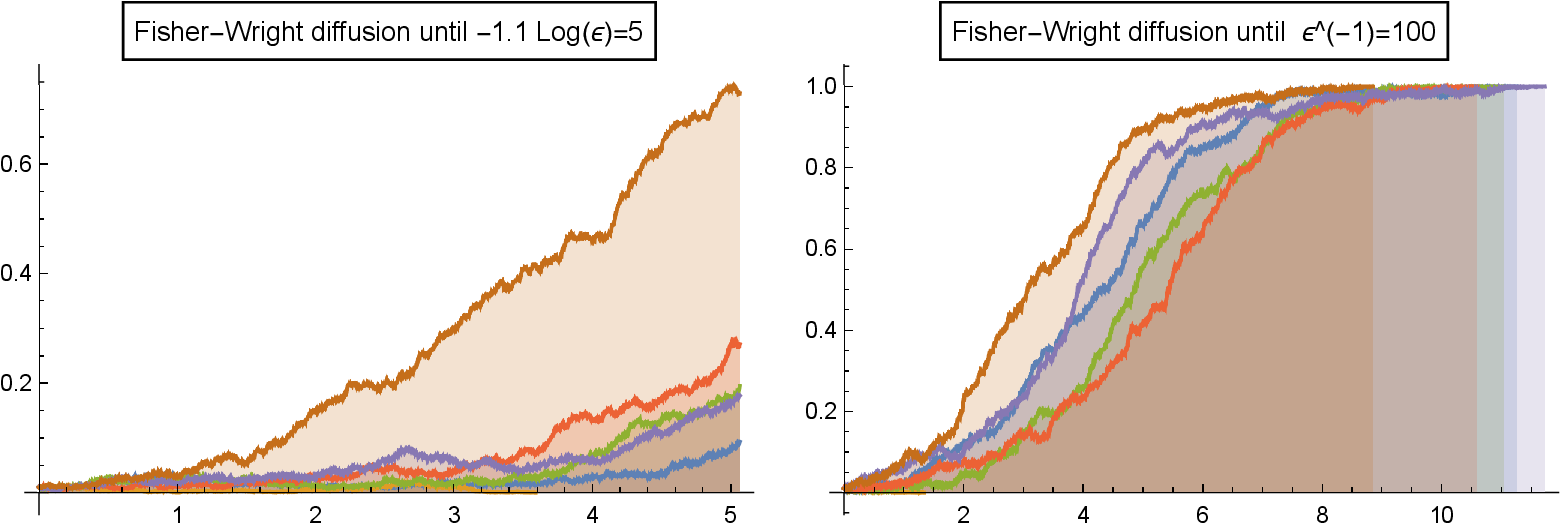}{6 paths of the  Kimura-Fisher-Wright diffusion \(dX_t = \g X_t (1 - X_t) dt + \sqrt{\eps X_t(1-X_t)} dB_t\), where $x_c=1$ is an  exit boundary, with $\eps=.01$. On the right,  three stages of evolution may be discerned}{0.7}

  \BEN
 \im  In the first stage,  the process leaves the neighborhood of the unstable point. The linearization of the SDE implies that here  a Feller branching approximation may be used, and this produces a certain exit law $W$ which will be carried over to the next stage as a  (random) initial condition.
 \im In the second ``semi-deterministic stage" (meaning that paths cross very rarely here),  the system moves towards its first  stable critical point $x_c$, following the trajectories of its fluid limit \eqr{detmain}, again over a time whose length converges to $\I$. A further renormalization produces  here the main result, the limit exit law \eqr{t0}.
   \im In the third stage, after the SDE has approaches the stable critical point
  of the fluid limit, ``randomness is regained" -- see crossings of paths in figures   \ref{f:FW2} and  \ref{f:lF2}); (if the  process may reach and overshoot the stable critical point, convergence towards
  a stationary distribution may occur).
  \EEN

The following result  was obtained   first in \cite{BCHK}, for the "Kimura-Fisher-Wright"  diffusion,
and extended subsequently to diffusions with bounded volatility.

\begin{thm}\label{mainT} {\bf Fluid limit with random initial conditions} \cite{BCHK}. Let $X^\eps_t$ satisfy Assumption \ref{a:A}, \eqref{main}, and $X^\eps_0=\eps>0$.
Suppose in addition
that
 the   diffusion coefficient $\sigma(\cdot)$ is continuous and {\bf  bounded},
 as well as its  first derivative,
and that $\m(\cdot)$ satisfies the following drift condition:
\begin{equation*}\label{drift}
\big|\m(y)-\m(x)\big|\leq \m'(0) |y-x|, \quad x,y\in \R_+.
\end{equation*}
Let $Y_t$ denote the solution to the {\bf scaled linearized equation}
\begin{equation}\label{Ylin}
d Y_t= \m'(0) Y_t dt + \sqrt{a'(0) Y_t} dB_t, Y_0=1 \Lra Y_t =  1 + \int_0^t \m'(0) Y_s ds + \int_0^t \sqrt{a'(0) Y_s} dB_s,
\end{equation}
known as  {\bf Feller branching diffusion}.

Then, \ith:
\BEN \im[(A)]  \begin{equation}\label{t0}
 X^\eps_{T^\eps} \xrightarrow[\eps\to 0]{\P} \T \phi(W),
\end{equation}
where
\BEN \im[(i)]
the random variable $W$ is the a.s.  martingale limit
\begin{equation}\label{W}
W:=\lim_{t\to\infty} e^{-\m'(0)t} Y_t= 1+\int_0^\infty e^{-\fr {\m'(0)} 2 s} \sqrt{a'(0) Y_s}dB_s
\end{equation}
\im[(ii)] $\T \phi(x)$  denotes the {\bf  limit of the deterministic flow pushed first backward in time by the linearized deterministic flow} $\phi_{t,lin}(x)=x e^{\m'(0) t}$ near the unstable critical point $0$
\begin{equation}\label{Hxlim}
\T \phi(x) =\lim_{t\to\infty} \phi_t\big(\phi_{-t,lin}(x)\big)
=  \lim_{t\to\infty} \phi_t\big(xe^{-\m'(0) t}\big), \quad x\ge 0.
\end{equation}

\EEN

\im[(B)]  Also, for any $T>0$,
\begin{equation}\label{newapp}
\sup_{t\in [0,T]}\big| X^\eps_{T_{\eps} + t} -x_t \big| \xrightarrow[\eps \to 0]{\P} 0,
\end{equation}
where $x_t$ is the solution of \eqref{detmain} subject to the initial condition
$
X_0 = \T \phi(W)$.
\EEN

\end{thm}
\beR
Note that $W$ depends only on the local
parameters $\m'(0), a'(0)$ of the diffusion  at the critical point. Assume from now on, without loss of generality that $a'(0)=1$ (recalling however that this is the only part of the stochastic perturbation that survives in the limiting regime), and let
\begin{equation} \la{g} \g:= \m'(0)>0 \end{equation}
denote the Malthusian parameter.

In the one -dimensional case, the \LT \ of $W_{\infty}$ is \wk \cite{Pardoux} and easy to compute. Indeed, letting  $u_t(\l)=-\fr 1 x \log(E_x[e^{-s Y_t}]$ denote the cumulant transform of this branching process,
 and solving the Riccati-type equation
\[\pd{u(t)}{t}=\g u(1-\fr{a'(0)}{2 \g} u), \] yields an explicit expression:
\begin{equation} \la{FLT}
E_x e^{-s Y_t} = \exp \left(- \frac{x s  e^{\g t}}{1+\frac{s a'(0)} {2\g }
(e^{\g t}-1)}\right), \quad s>0
\end{equation}see, e.g.,  \cite[Ch 4.2, Lem. 5, pg. 24]{Pardoux}.

One may conclude  
from the explicit \eqr{FLT} that
\begin{equation} \la{LTF}
E_1 e^{-s W_{\infty}}=\lim_{t \to \I} u_t(s e^{-\g t}) =
\exp \left(-\frac{2  s \g / a'(0)   }{2\g/ a'(0) +s    }\right),
\end{equation}and one may check that $W_{\infty}$  is  a Poisson sum with parameter $ {2\g/a'(0) }$ of independent \expo random variables \be W_{\infty}=\sum_{j=1}^{N_{2\g }} \tau_j, \tau_j\sim \mathrm{Expo}(2\g/a'(0) ). \ee
 \eeR
 \beR  Computing the limit $\f(s):=E e^{-s W_{\infty}}=\lim_{t\to\infty}E e^{-s e^{-\g t} Y_t }$  is a famous problem in the theory  of supercritical
branching \procs.
Recall that \BEN \im
For Galton-Watson \procs, $\f(s)$ \sats the {\bf Poincar\'e - Schroeder functional equation}\be \la{Sch} \f(m s)= \H p (\f(s)), m= \H p'(1) \ee where $\H p(s)$ is the probability generating function of the progeny  \cite[I.10(5), Thm I.10.2]{AN}.
\im For continuous time \br \procs, letting $\Psi(s)=\H p(s)-s$ denote the branching mechanism, and $\th(s)=\f(s)^{-1}$ denote the functional inverse,
\ith
\be \th(1)=0, \fr{\th'(s)}{\th(s)}=\fr{\Psi'(1)}{\Psi(s)}, 0\leq s \leq  1  \Lra \th(s)=(1- s) e^{-\int_s^1 ( \fr{\Psi'(1)}{\Psi(u)}+\fr 1 {1-u}) du}, 0\leq s \leq  1\ee
see \cite[III.7(9-10), p.112]{AN} and
\be \la{ctP} s \f'(s)=\Psi'(1)^{-1}\Psi(\f(s)), \f(0)=1.\ee

\Fe for binary splitting with branching mechanism $\Psi(s)=  s^2-s$, we find
\begin{equation*} s \f'(s)=\f(s)(\f(s)-1)\Lra \f(s)=\frac{1 }{1 +s     }, \th(s)=\fr{1-s} {s}\end{equation*}
with  $W_\I$   \expo with parameter $1$,
and for geometric branching with \para $1-u$ we find
\begin{equation*} \th(s)^{1- 2u}=\fr{(1-s)^{\fr 1{u}}} {(1- u(1+ s))^{\fr 1{1-u}}} .\end{equation*}
The example of $k$-ary fission is also explicit-- see \cite[p. 218]{bingham1988limit} and
\cite[p. 119]{karlin1968embeddability}.

\beQ Extend the results of \cite{BCK} from birth-death to Markov discrete
space with finite number of transitions upwards and downwards.  Solve numerically the Schroeder equation.
\eeQ

\im

For the continuous state case, letting $-\kappa(s)=\ln\Big(E[e^{-s W_{\infty}}]\Big)$ denote the logarithm of the \LT
and  $\Psi(s)$ denote the branching mechanism, \ith
\be \la{difP} s \k'(s)=\Psi'(0)^{-1}\Psi(\kappa(s))\ee
see \cite[Cor. 4.3]{bingham1976continuous} and also the Appendix, for the multi-type case.

 Also \cite[Thm 4.2]{bingham1976continuous}, \ith  the functional inverse
$\th(s)=\kappa(s)^{-1}$ satisfies
\be \fr{\th'(s)}{\th(s)}=\fr{\Psi'(0)}{\Psi(s)}, 0\leq s \leq  \Psi'(0)  \Lra \th(s)= s e^{\int_0^s ( \fr{\Psi'(0)}{\Psi(u)}-\fr 1 u) du}, 0\leq s \leq  \Psi'(0).\ee

\Fe for the Feller branching diffusion with branching mechanism $\Psi(s)= \g s- \fr 1 2 s^2$, we find
\[\th(s)= \frac{2 \gamma  s}{2 \gamma -s}, \kappa(s)=\frac{2\g s    }{2\g +s     }.\]

\EEN

 \eeR
\beR  The main part of Theorem \ref{mainT} is the equation \eqref{t0} which identifies the limit after
the second stage
\begin{equation} X^\eps_{T^\eps}=\Phi_{T^\eps}^\eps(\eps) \xrightarrow[\eps\to 0]{\P} \lim_{t\to\infty} \phi_t\big(\phi_{-t,lin}(W)\big)=\T \phi(W), \la{tph}
\end{equation}$\Phi_{t}^\eps(x)$ denotes the flow generated by
the SDE \eqref{main}.

Note that $\T \phi$ depends only on the dynamical system $\m$.
By \cite[Prop. 4.1]{BCHK},  it is a nontrivial solution of the ODE \begin{equation} {\m'(0) x} \T \phi'(x)=   \m(\T \phi(x)),  \T \phi(0)=0 \la{tph1}\end{equation}
which is equivalent to the  {\bf Poincar\'e functional equation}
\be \la{tph2} \T \phi(x e^{\g t}) =\phi_t(\T \phi(x)) \Eq   \m(\T \phi(x))=\T \phi(\g x) \ee
 arising in Poincar\'e conjugacy relations for dynamical systems.
Interestingly, this is the same type of equation as \eqr{difP}, minus the restriction that $v \cd$ be a Bernstein function.

The inverse $w(x)=\T \phi(x)^{-1}$ when $\T \phi(x)$
\sats \eqr{tph1} is given by
\begin{equation}\label{Gdef}
w(x) = x e^{\int_0^x \left(\frac \g{\m(u)}-\frac 1{ u}\right)du }, 0\leq x \leq \g.
\end{equation}

 \eqr{tph1}, \eqr{tph2} suggest  possible generalizations to multidimensional diffusions
 (and possibly to jump-diffusions (where a  CBI might replace  the Feller diffusion in the limit).



\eeR

\beR Part 2. of Theorem \ref{mainT} follows immediately    by a simple change of time:   letting $\widetilde X^\eps_t=X^\eps_{T^\eps+t}$, and $\widetilde{B}_t = B_{T^\eps + t} - B_{T^\eps}$
one obtains from \eqref{main}
\[
\widetilde X^\eps_t =\widetilde X^\eps_{0} +\int_0^t f(\widetilde X^\eps_s)ds+
 \int_0^t \sqrt{ \eps \sigma(\widetilde X^\eps_s)}d\widetilde{B}_s,
\]
and the result follows from  \eqref{t0} by the fluid convergence Theorem \ref{t:Kurtz}. This part may be viewed as describing  ``short transitions" (invisible on a long time scale) between the second and third stages.
\eeR

\beR  \la{r:ph} The limit \eqr{t0} describing the position after the second stage  has  been established in \cite{BCHK} for
  one dimensional distributions with bounded $\s(x)$.    This assumption seems \how restrictive, since for typical diffusions whose fluid limit $\phi_t(x)$ admits a stable critical point $x_c$, the  probability  of leaving the neighborhood of the  stable point $x_c$ is very small as $\eps \to 0$. This intuition is confirmed by simulations  --see  Figure  \ref{f:lF2}.

\eeR

The remark \ref{r:ph} suggests the relation of our problem to that of  studying the maximum of $X_t$.  

More precisely, we would like  to establish and exploit the plausible fact that $\for \th >1$}
\begin{equation} \lim_{\eps \to 0} P[ T_{\th x_c }< T^\eps |X_0=\eps]=\lim_{\eps \to 0} P[\sup_{0\leq t\leq  T^\eps} X^\eps_t>\th x_c |X_0=\eps]=0, \la{key}\end{equation}where $x_c$ is the closest  critical point
 towards which the diffusion is attracted, and $T_{\th x_c }$ is the hitting time of ${\th x_c }$;  clearly, \eqr{key} renders unnecessary the assumption
that the diffusion coefficient $\sigma(\cdot)$ be  bounded.

 A weaker statement than \eqr{key}, but still sufficient for a slight extension,   is provided in the elementary Lemma \eqr{det1} below.

{\bf Contents}. The  paper is organized as follows. In Section \ref{s:ext} we offer,  based on   Lemma \ref{det1},  a slight extension of Theorem \ref{mainT} of \cite{BCHK}. A  conjecture (see Problem \ref{c:1}) is presented here as well. We  illustrate our new result with the  example  of the logistic Feller diffusion in Section \ref{s:LF}. We include for convenience
in Section \ref{s:sk} an outline of the remarkable paper \cite{BCHK}.
\section{An extension of Theorem \ref{mainT} \cite{BCHK} \la{s:ext}}

Recall now from \cite{BCHK}
that   the restrictive condition $\|\sigma\|_\infty < \I$ is used for proving that\fn[3]{Let us recall the proof of this important piece of the puzzle. Let $\Phi_{s,t}(x)$, $\phi_{s,t}(x)$ denote the stochastic and deterministic flows  generated respectively by the SDE \eqref{main} and
ODE \eqref{detmain},  put $\Phi^\eps_t :=  \Phi_{t_c,t_c+t }(X^\eps_{t_c})$, $\phi_t := \phi_{t_c,t_c+t }(X^\eps_{t_c})$ for brevity,
and define  $\delta^\eps_t = \Phi^\eps_t-\phi_t$.  Subtracting equations \eqref{main} and \eqref{detmain}  and applying the It\^o formula:
\begin{align*}
E \big(\delta^\eps_t\big)^2  =\, &  E \int_0^t 2\delta_s \big(\m(\Phi^\eps_s) - \m(\phi_s)\big)ds
+\int_0^t \eps E   \sigma(\Phi^\eps_s) ds  \leq
 \int_0^t 2 \g E (\delta_s)^2  ds  +  \eps t \|\sigma\|_\infty, t \in \R_+
\end{align*}
where  assumption \eqref{drift} was used.
By Gr\"{o}nwall's inequality
\begin{align} \la{Gro}
E \Big(\Phi_{t_c,t_1}(X^\eps_{t_c}) - \phi_{t_c, t_1}(X^\eps_{t_c}) \Big)^2 =\;  &
E \big(\delta^\eps_{t_1-t_c}\big)^2 \leq
C_1 \eps t_1 e^{2\gamma(t_1-t_c)}\leq C_2  \eps^{2 c-1} \log \frac 1 \eps \xrightarrow[\eps\to 0]{} 0
\end{align}
where the convergence holds since $c\in (\frac 1 2, 1)$.}

\begin{equation} \la{eBC} \|\sigma\|_\infty < \I, c \in (1/2,1) \Lra
  \Phi_{t_c,t_1}(X^\eps_{t_c}) - \phi_{t_c, t_1}(X^\eps_{t_c})    \xrightarrow[\eps \to 0]{L^2} 0,
\end{equation}where $t_c=c T^\eps$.

We will show now that it is possible to remove the condition $\|\sigma\|_\infty < \I$ in \eqr{eBC},  if only convergence in probability is needed,  by  assuming rather weak and natural conditions on the scale function $s\cd$. Recall that the scale function $s$ is  defined (up to two integration constants) as an arbitrary increasing solution of the equation $\LL s(x)=0$, where $\LL$ is the generator operator of the diffusion, and that this  function is continuous -- see \cite[Ch. 15, (3.5), (3.6)]{KT}
 (noting that \cite{KT} denote the scale function by $S\cd$).

\begin{lem} \la{det1} Assume that $0$ is an attracting boundary
 and that $r$ is an unattracting  boundary, i.e. that $, s(0_+)>-\I, s(r-)=\I$. 
Put
\begin{equation} \ovl X^\eps=\sup_{0\leq t <\I} X^\eps_t,\end{equation}where $X^\eps$ is defined in \eqr{main}. Then:

 \begin{equation} \la{M1} \text{(A)   } \quad \for \eps, \lim_{M \to r} P_\eps[\ovl X^\eps > M] = \lim_{M \to r} \fr{s(\eps)-s(0)}{s(M)-s(0)}= (s(\eps)-s(0))\lim_{M \to r} \fr{1}{s(M)-s(0)}=0,\end{equation} 
and
\begin{equation} \la{eps1}  \text{(B)   } \quad c \in (1/2,1) \Lra \Phi_{t_c,t_1}(X^\eps_{t_c}) - \phi_{t_c, t_1}(X^\eps_{t_c})    \xrightarrow[\eps \to 0]{P} 0. \end{equation}\end{lem}

\begin{proof}
 \eqr{M1}  is straightforward. Indeed, recall that the boundary $0$ is attracting. Then,
 \begin{equation} P_\eps[\ovl X^\eps > M] = P_\eps[T_M <T_0]= \fr{s(\eps)-s(0)}{s(M)-s(0)} \end{equation} where $T_{0}, T_M$ are the hitting times of $X_t^\e$ at $0$ and $M$ -- see
 \cite[Ch. 15, (3.1), (3.10)]{KT}.
 Using now  the continuity of the scale function $s\cd$ \cite[Ch. 15, (3.5), (3.6)]{KT}
 (note that \cite{KT} denote the scale function by $S\cd$) yields $\lim_{M \to r} s(M)=s(r_-)=\I$ and the result.

 \eqr{eps1} follows by a similar argument. Indeed, denote the  deterministic and stochastic flows generated by the
ODE \eqref{detmain} and SDE \eqref{main} (i.e. the solutions of these equations at time $t$ that start at $x$ at
time $s$)  by $\phi_{s,t}(x)$ and $\Phi_{s,t}(x)$, respectively, and put $\Phi^\eps :=  \Phi_{t_c,t_1 }(X^\eps_{t_c})$ and $\phi^\eps := \phi_{t_c,t_1 }(X^\eps_{t_c})$ for brevity
and define  $\delta^\eps = \Phi^\eps-\phi^\eps$.   For fixed $\eps $ and $M$, \ith
 \begin{align*} \for \d >0, P_\eps[|\delta^\eps| >\d] &\leq P_\eps[\ovl X_{T^\eps}^\eps \leq M] P_\eps[|\delta^\eps| >\d|\ovl X_{T^\eps}^\eps \leq M] +
P_\eps[\ovl X_{T^\eps}^\eps > M]\\&\leq P_\eps[\ovl X_{T^\eps}^\eps \leq M] P_\eps[|\delta^\eps| >\d|\ovl X_{T^\eps}^\eps \leq M] +
P_\eps[\ovl X^\eps > M].
 \end{align*}
 Letting now $\eps $ to $0$ makes the first term go to $0$ by \eqr{eBC}, yielding
 \[
   \for M <r, \for \d >0, \limsup_{\eps \to 0} P_\eps[|\delta^\eps| >\d] \leq \lim_{\eps \to 0} \fr{s(\eps)-s(0)}{s(M)-s(0)}=0\]
 where we have used again the continuity of the scale function.

\end{proof}

\beT \la{tn} The conclusions of Theorem \ref{mainT} still hold under the assumptions of Lemma  \ref{det1}.
\eeT
\begin{proof}  Theorem \ref{mainT} of \cite{BCHK} only uses the assumption
$\|\sigma\|_\infty < \I$ in establishing the unnecessarily strong result \eqr{eBC}.  Providing weaker conditions for the weaker but still sufficient result \eqr{eps1} establishes therefore  our claim.

\end{proof}

\beQ \la{c:1} Note that essential use of $s(0) > -\I$ was made in \eqr{M1}. We conjecture however that a finer analysis will reveal that the result of  Theorem \ref{tn} still holds whenever $r$ is ``repelling/unattracting", more precisely when it is natural unattracting or entrance, cf. Feller's classification
of boundary points \cite[Ch. XV]{KT}.

\eeQ

\sec{Examples with $\lim_{t \to \I} X_t/x_t = 0$: The logistic
 Feller and Gilpin-Ayala diffusions \la{s:LF}}

 We recall now some famous examples for which the conditions of our Lemma \ref{det1} 
hold.
 The logistic
 Feller diffusion is defined by
\[
dX_t = \g X_t\Bigg(1  -  \; \fr{X_t}{x_c} \Bigg) dt + \sqrt{\eps X_t} dB_t, \; X_t \in (0,\I).
\]

The limit point $x_c$ of $x_t$    is a regular point for the diffusion; \wlg we will take it equal to $1$.  The scale density  $s'(x)=e^{-\fr{2\g}{\eps}(x- \fr{x^2}2)}$ is integrable at  $0$, but not at $\I$, and the speed density \cite{KT} $m'(x)=\fr{e^{\fr{2\g}{\eps}(x- \fr{x^2}2})}{\eps x}$  is integrable at  $\I$, but not at $0$, so that the conditions of Lemma \ref{det1} hold.\fn[4]{\Frt conform Feller's boundary classification \cite{KT},  $0$ is an exit boundary
since $s'(x) m[x,1]$ is integrable at $0$, and absorbtion in $0$ occurs  with  \proba \; 1, and $\I$
is an entrance (nonattracting) boundary,
since $m'(x) s[1,x] 
$ is integrable at $\I$--see also \cite{Cattiaux,bansaye2017diffusions} and \cite{foucart2017continuous} for the generalization to continuous-state branching processes with competition.}

 \Thr  fluid convergence with random initial point   before $T_\eps$ \cite{BCHK} still holds,  
 with the same  deterministic   flow and  random initial condition as for the Kimura-\FiW \; diffusion studied in \cite{BCHK}
\[
\phi_{t}(x) = \frac{xe^{\g t}}{1-x+xe^{\g t}}, \; \T \phi(x) =\frac{x}{1+x}, X_0   =\frac W{W+1}
\]
(since $\m(.),a'(0)$ did  not change)--see Figure \ref{f:lF2}.

\figu{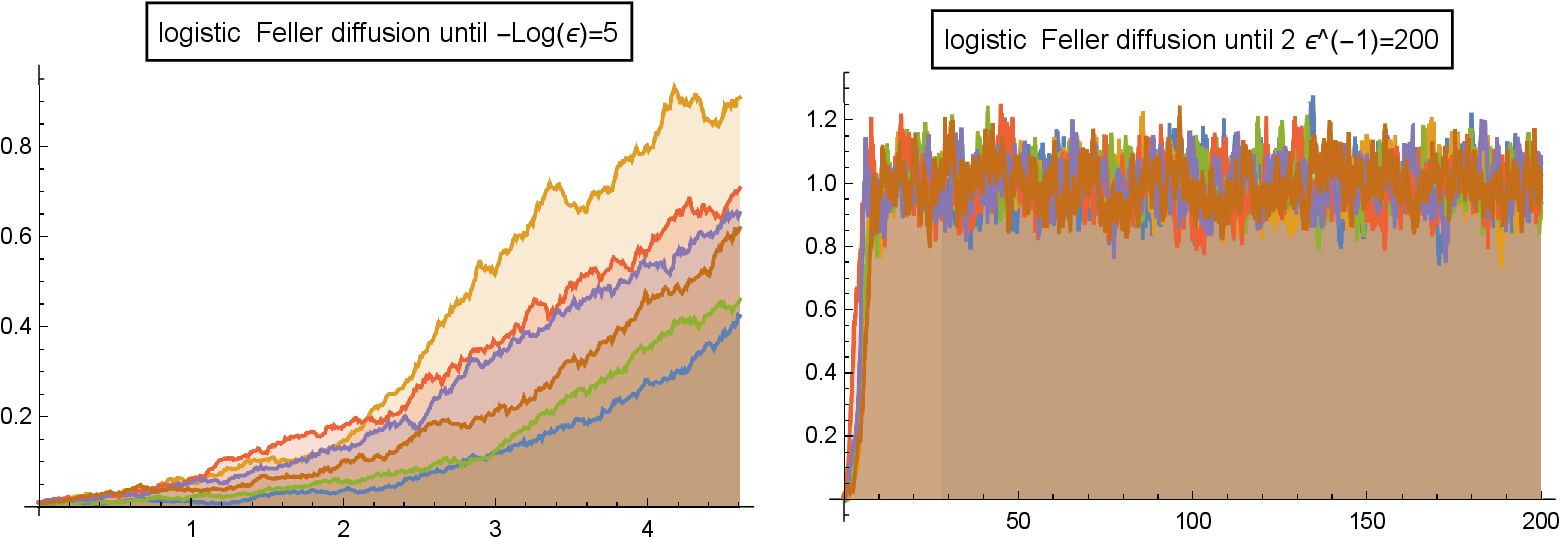}{6 paths of the logistic Feller  diffusion ($x_c=1$ is regular) with $\eps=.01$, until $T_\eps$ and after}{0.7}

\figu{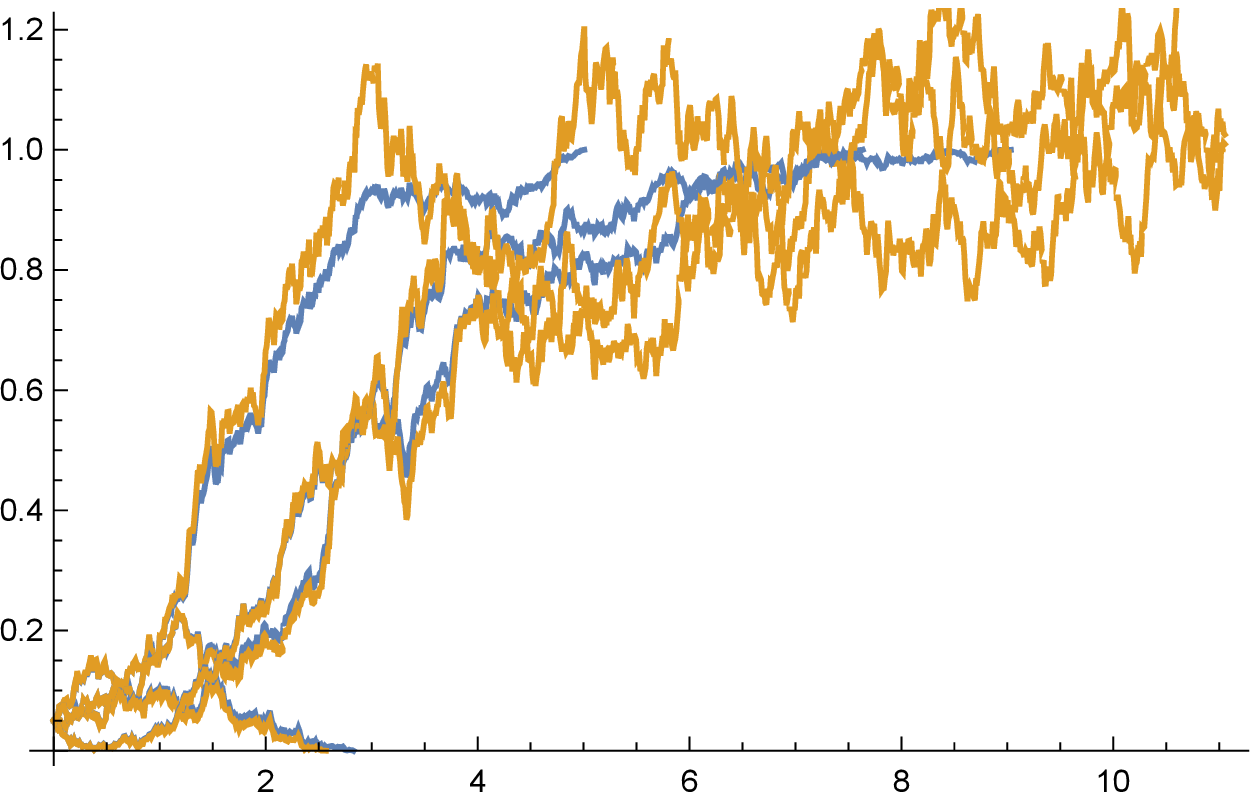}{6 paths of the logistic Feller and Kimura-Fisher-Wright diffusions with $\eps=1/20$, before and after $T_\eps$}{0.6}

In fact, the paths of the logistic Feller and Kimura-Fisher-Wright diffusions are almost indistinguishable up to $T^\eps$ of each other --see Figure  \ref{f:LWL}. 
 After reaching the neighborhood of $x_c$ however, the paths  split, reflecting  the different  natures (regular and exit) of $x_c$ for these two stochastic processes.

 Some  other examples of interest in population theory are the diffusion processes defined by
  the SDEs
 \begin{align*}  & d X_t = \g X_t \Big(1  -  \;(\fr{X_t}{x_c}^\th \Big) dt +  \s \sqrt{ X_t} d B_t, \s>, \th  >0,
 \\& d X_t = \Big[\g X_t \Big(1  -  \;\fr{X_t}{x_c} - \b \fr{X_t^{n-1}}{1+X_t^{n}}\Big)\Big] dt +  \s \sqrt{ X_t} d B_t, \b \geq 0, n\geq 1,
 \end{align*}
which are stochastic extensions with square root volatility of \det \ population \ models
    introduced by  Gilpin and Ayala and Holling respectively. 

 {It is easy to check that adding the exponents $\th$ and $n$
does not affect integrability  of the scale and speed densities of  these diffusions, so that
our extension applies.
\Frt the rescaled flow $\T \phi$ may be  computed numerically by \cite[Prop. 4.1]{BCHK} (and even symbolically for small integer values of $\th, n$).

Moving away from the square root volatility case, an interesting, still open question is to investigate whether analogues of the \cite{BCHK} result are available for the
processes \satg
$ d X_t =\g  X_t \Big(1  -  \; (\fr{ X_t}{x_c})^\th  \Big) dt + \se ( X_t)^\a d B_t, \qu \a >0.$\fn[4]{The  particular case $\a= \th =1$ is the famous Verlhurst-Pearl diffusion (VP)-- see for example \cite{liu2015sufficient}.}
\section{Sketch  of the proof of Theorem \ref{mainT} \cite{BCHK} \la{s:sk}}

Recall that  $\displaystyle t_c=c t_1 $ with  $c\in (1/2,1)$, arbitrary, and note
that
$X^\eps_{T^\eps} =\Phi_{t_c,t_1}(X^\eps_{t_c})=\Phi_{t_c,t_1}(\Phi_{t_c}(\eps))$.
The idea of the proof is to approximate this \rv by
\begin{equation}  X^\eps_{T^\eps} \approx \phi_{t_c,t_1}(\Phi_{t_c}(\eps))\xri{\eps \to 0} \T \phi(W),\end{equation}with the random variable $W$ from \eqref{W}.

The proof of \cite{BCHK} involves several steps
\BEN
\im The first idea for establishing the approximation $\T \phi(W)$ of $X^\eps_{T^\eps}$  is to  {\bf blow-up } the process near the boundary $0$
\[\T X^\eps_t : =\eps^{-1} X^\eps_t,
\] which fixes the initial condition to $1$ and
 changes the SDE to
\begin{equation}\label{mainsc}
d \T X^\eps_t =  \eps^{-1} \m(\eps \T X^\eps_t) dt + \sqrt{\fr{ a(\eps \T X^\eps_t)}\eps} dB_t,\quad t\ge 0,
\end{equation}
\iec that a subsequent {\bf linearization of the SDE} yields
\[ \T X^\eps_t \approx Y_t\]
where $Y_t$ is a
{\bf Feller branching diffusion}  started from $1$, defined by
\begin{equation}\label{Feller}
Y_t=1+\int_0^t \m'(0) Y_sds+\int_0^t\sqrt Y_s dB_s, \quad t\ge 0.
\end{equation}
One may take advantage then of  the well-known \nne martingale convergence theorem for the ``scaled final position" of the branching process $Y_t$
\begin{equation} W:=\lim_{t \to \I} e^{- \m'(0) t}  Y_{t}.\la{LTF} \end{equation}\beR
Let us note that the linearization for processes satisfying $a(x)=O(x^2)$ and failing Assumption \ref{a:B}, like the linear Gilpin-Ayala \eqr{lGA}, leads to geometric Brownian motion. In this case,  \eqr{LTF}
holds with $W=0$, and a different approach seems necessary.
\eeR

\im After ``blowing up" the beginning of the path, the second idea is to  {\bf ``look from far away"}. We want to    break
 the trajectory at a suitably chosen time point
 \begin{equation} t_c < t_1=T^\eps=\displaystyle\frac{1}{\g} \log\frac 1 \eps
 \end{equation}such that
 before $t_c$, the original process  is close to Feller's branching diffusion \eqr{Feller}, and convergence to the limit $W$ of the Feller diffusion occurs, i.e.
\begin{equation} \la{e:ess} X^\eps_{t_c}=\eps \T X^\eps_{t_c} =e^{-\g t_1} \T X^\eps_{t_c} \approx e^{-\g t_1} Y_{t_c} = e^{-\g (t_1-\t_c)} e^{-\g t_c} Y_{t_c}\approx e^{-\g (t_1-\t_c)} W. \end{equation}The first approximation
 $e^{-\g t_c} \T X^\eps_{t_c}   \xrightarrow[\eps\to 0]{L^1}Y_{t_c}$
 follows from the following lemma \cite{BCHK} showing that the solution of \eqref{main} converges, under appropriate scaling,
to the Feller branching diffusion \eqref{Feller}.

\begin{lem} \label{L1}
Let $\T X^\eps_t : =\eps^{-1} X^\eps_t$, where $X^\eps_t$ is the solution of \eqref{main} subject to $X^\eps_0=\eps$. Then
\[
\T X^\eps_t \xrightarrow[\eps\to 0]{L^1} Y_t, \quad \forall\, t\ge 0,
\]
where $Y_t$ is the solution of \eqref{Feller}.
\end{lem}

  Putting these together yields
 \(
 \phi_{t_c, t_1}(X_{t_c}^\eps)\xrightarrow[\eps\to 0]{\P} \T \phi(W).
 \)

 \im  The hardest part is proving that  in the
second portion $[t_c,t_1]$, the influence of the stochasticity is negligible,
\fe that  $
  \Phi_{t_c,t_1}(X^\eps_{t_c}) - \phi_{t_c, t_1}(X^\eps_{t_c})    \xrightarrow[\eps \to 0]{L^2} 0
$, as proved in  \cite{BCHK} under the restrictive assumption $\|\sigma\|_\infty < \I$.

\EEN

Putting it all together in one line, one must prove that
 \begin{equation} X^\eps_{t_1}=\Phi_{t_c, t_1}(X_{t_c}^\eps)\approx \Phi_{t_c, t_1}(W e^{-\g (t_1-\t_c)})\approx \phi_{t_c, t_1}( W e^{-\g (t_1-\t_c)})\xrightarrow[\eps\to 0]{\P} \T \phi(W).\end{equation}
To extend \cite{BCHK}, it is sufficient to improve the third \app \; step above.

{\bf Acknowledgement:} We thank  J.L. Perez for useful remarks and the referee for the help in improving the exposition.

\bibliographystyle{acmurl}

\begin{address}
  Florin Avram and Jacky Cresson\\
  CNRS / UNIV PAU \& PAYS ADOUR/LMAP - IPRA, UMR5142  \\
  64000, PAU, FRANCE \\
  \texttt{Florin.Avram@orange.fr} and \texttt{jacky.cresson@univ-pau.fr}
\end{address}

\end{document}